\documentclass[a4paper,reqno,12pt]{amsart}
\usepackage[T2A]{fontenc}
\usepackage{amsfonts, amssymb, amsmath}
\theoremstyle{plain}
\newtheorem{proposition}{Proposition}

\newtheorem{theorem}{Theorem}
\theoremstyle{remark}

\theoremstyle{definition}
\newtheorem{example}{Example}
\newtheorem{question}{Question}
\newtheorem{definition}{Definition}
\renewcommand{\le}{\leqslant}
\renewcommand{\ge}{\geqslant}
\newcommand{\Le}{\leq}
\newcommand{\Ge}{\ge}
\DeclareMathOperator {\ord}{ord}
\DeclareMathOperator {\Card}{Card}
\DeclareMathOperator {\rk}{rk}
\DeclareMathOperator{\Char}{char}

\newcommand{\N}{\mathbb{N}}

\newcommand{\Z}{\mathbb{Z}}
\newcommand{\K}{\mathbb{K}}
\def\deg{\operatorname{deg}}
\def\ord{\operatorname{ord}}
\def\C{\mathbb C}

\def\F{\mathcal {F}}
\def\G{\mathcal {G}}

\begin{document}
\title[]{A Bound for a Typical Differential Dimension\\ 
of Systems of Linear Differential Equations.} 
\author{M. V. Kondratieva} 
\address{Moscow State University\\ 
Department of Mechanics and Mathematics\\ Leninskie Gory, Moscow, 
Russia, 119991.}
\email{kondratieva@sumail.ru} 
 
\begin{abstract} We prove upper and lower bounds for
 leading coefficient of Kolchin dimension
 polynomial of systems of partial linear differential equations
in the case of codimension two, squared
by the orders of the equations  in the system.
A notion of typical differential dimension plays 
an important role in
differential algebra, some its estimations
were proved by J.Ritt and E.Kolchin, also they
advanced 
several conjectures that were later
refuted. Our bound generalizes
the analogue of B\"ezout theorem, which has been proved in~\cite{KLMP}
for one differential indeterminate. It is better, than 
estimation, proved by  D.Grigoriev in~\cite{Grig}.


\end{abstract}

\maketitle
\thispagestyle{empty}

\section{Introduction}
One of the basic objects of stude in differential algebra is the
differential dimension polynomial introduced by E.Kolchin~\cite{Kolchin}. 
This is an analogue of dimension in algebraic
geometry, and estimations of coefficients of the dimension polynomial
are  classical unsolved problems of differential algebra.
If the characteristic set  of a prime differential
ideal in the ring of differential polynomials is known,
the task of finding of
the dimension polynomial is
 a simple combinatorial problem, which can be solved
by a number of  algorithms (see, for example, \cite{KLMP}).
It follows from these algorithms, that the dimension polynomial's
leading coefficient polynomially
depends on orders of elements in the characteristic set  of a prime differential component. 
But for finding  characteristic sets we apply
algorithms (different variants of Rosenfeld-Gr\"obner's
algorithm,
see for example~\cite{Boul, Hub}),
complexity of which 
are unknown. In case of linear equations
a characteristic set is Gr\"obner basis of  module of K\"ahler differentials, 
and in \cite{Janet} was proved an upper double exponential
estimation on orders of elements of characteristic set. This result 
generalizes  an upper bound  for  orders of Gr\"obner basis 
of polynomial ideal (see, for example, \cite{Dube}). We will not use
for the estimation of leading coefficient of dimension polynomial 
a bound for elements of characteristic set, and in case of 
codimension two we will get more exact,
than in~\cite{Grig} estimation.
Note, that it is known double exponential lower bound for
orders  Gr\"obner basis' elements (\cite{Mayr}),
but still there is not an analogical estimation 
for the coefficients of Kolchin dimension  polynomial. Therefore 
an example \ref{be1} is important, since 
it gives a squared lower bound  
in a codimension 2.

In the case of  nonlinear systems Kolchin was proved
an estimation of leading coefficient if a 
degree of dimension polynomial
is  less on 1 than  the number of differentiations (i.e. in case of codimension 1).
In \cite{KLMP} (see  p. 265) Kolchin's conjecture    
concerned with an estimation in a codimension more than 1 and 
some other polynomial hypotheses were refuted.
In \cite{Kondr} the rough estimation of leading coefficient 
at any value of differential dimension, including nonlinear case,
based on the Ackermann function   was proved.
Whether it is possible to prove an  upper
double exponential estimation for the nonlinear systems, still
is open problem.

\section{Preliminary facts.}

One can find 
basic concepts and facts  in~\cite{Kolchin, Ritt, KLMP}.

Denote   the set  of integers by $\Z$, 
non-negative integers by $\N_0$  and   binomial coefficients by $\binom{n}{k}$.
For $e=(j_1,\dots, j_m) \in\N_0^m$, the order of  $e$ is defined by 
$\ord e=\sum_{k=0}^mj_k$. 
Note that any numerical 
polynomial $v(s)$ can be written as $v(s) =\sum_i^da_i\binom{s+i}{i}$,
where $a_i\in\Z$. We call numbers $(a_d,\dots, a_0)$ {\bf
standard coefficients } of polynomial $v(s)$.

Now  we define the Kolchin dimension polynomial
of a subset $E\subset \N_0^m$.
Regard the following partial order on $\N_0^m$: the relation
$(i_1,\dots, i_m)\le(j_1,\dots, j_m)$ is equivalent to
$i_k\le j_k$ for all $k=1,\dots, m$.
We  consider a function $\omega_E(s)$, that   in a point $s$ 
equals $\Card V_E(s) $, where $V_E(s)$ is the set
of points $x\in\N_0^m$ such that $\ord x\le s$ 
and for every $e\in E$ the condition $e\le x$ isn't true.
Then (see for example, \cite{Kolchin}, p.115, or \cite{KLMP}, 
theorem 5.4.1 ) function $\omega_E(s)$ for all sufficiently large 
$s$ is a numerical polynomial. We call this polynomial
the Kolchin dimension polynomial of a subset $E$.

\begin{definition}\label{3.2.1} An operator $\partial$ 
on a
commutative ring
$\K$ with unit is called
{\bf  a derivation 
} 
if it is linear
$\partial(a+b) =\partial(a)+\partial(b)$ and the Leibniz's rule  
$\partial(ab) =\partial
(a)b+a\partial(b)$ holds for all elements $a, b\in \K$.

{\bf A differential ring} (or $\Delta$-ring) is a 
ring $\K$ endowed with a set of derivations
$\Delta=\{\partial_1,\ldots,\partial_m\}$
which commute pairwise. 

Construct the multiplicate monoid $$\Theta =\Theta(\Delta) = \left
\{\partial_1^{i_1}\cdot\ldots\cdot\partial_m^{i_m}\:\big|\: i_j\Ge 0,\ 1\Le j\Le m\right\}$$
of {\bf derivative operators}.

If $\theta= \partial_1^{i_1}\cdot\ldots\cdot\partial_m^{i_m}$
we  define {\bf order of derivative operator $\theta$}:
$$\ord(\theta) = i_1+\ldots+i_m\
\text {and}\
\Theta(r) =\{\theta\in\Theta| \ \ord(\theta) \le r\}.
$$ Let
$$R=\K\{y_j\:|\: 1\Le j\Le n\} := \K[\theta y_j\:|\: \theta\in\Theta, 1\Le j \Le n]\
$$
be a ring of commutative polynomials with coefficients in $\K$ in the 
infinite set of 
variables 
$\Theta Y=\Theta(y_j)_{j=1}^n$, and
$$R_r = \K\big[\Theta(r) y_j\big],\quad r\Ge 0.$$  
A ring $R$ is called a {\bf ring of differential polynomials
} in  differential indeterminate $y_1,\ldots, y_n$ 
over
$\K$.
\end{definition}
$R$ also is $\Delta-$ring.
Below we consider the case  when $\K$ is the differential field $\F$
and $\Char\F=0$ only. An ideal $I$ in $\F\{y_1,\ldots,y_n\}$ is 
called {\bf
  differential}, if $\partial
f\in I$ for all $f\in I$ and $\partial\in\Delta$ .

Let $\Sigma \subset \F\{y_1,\dots,y_n\}$ be a set of differential
polynomials. For the differential and radical differential
ideal generated by $\Sigma$ in $\F\{y_1,\dots,y_n\}$, we use
notations $[\Sigma]$ and $\{\Sigma\}$, respectively.
\begin{definition}\label{ran}
A {\bf ranking} 
is a total order $>$ on the set $\Theta Y$
satisfying the following conditions: for all $\theta\in\Theta$ and
$u,v\in\Theta Y$:
\begin{enumerate}
\item $\theta u \Ge u,$
\item $u \Ge v \Longrightarrow \theta u \Ge \theta v.$
\end{enumerate}
A ranking $>$ is called
{\bf orderly} if $\ord u > \ord v$ implies $u > v$ for all
derivatives $u$ and $v$. 

\end{definition}
A differential polynomial $f\in R$ is called {\bf linear}, 
if its degree (as a polynomial
in  variables $\theta y_j\:|\: \theta\in\Theta, 1\Le j \Le n$) is equal to 1.
A system  $\Sigma$  is called the 
system of linear differential equations, 
if every element $\Sigma$ is linear.
Let $u$ be a derivative,
that is, $u = \theta y_j$ for
$\theta = \partial_1^{i_1}\cdots\partial_m^{i_m} \in \Theta$
 and $1\Le j \Le n$. The {\bf order} of $u$ is defined as
$$\ord u=\ord\theta=i_1+\ldots+i_m.$$ If $f$ is a differential
polynomial, $f\not\in\F$, then $\ord f$ denotes the maximal order of
derivatives appearing effectively in $f$.

\begin{definition}\label{3.2.38}
Let $\F$ be a differential field with a set of derivations
$\Delta = \{\partial_1,\ldots,\partial_m\}$.
The ring $D=\F[\partial_1,\dots,\partial_m]$ of skew polynomials in 
indeterminates $\partial_1,\dots,\partial_m$ with coefficients 
in $\F$ and the commutation rules 
$\partial_i \partial_j=\partial_j \partial_i,\ \partial_i a=
a\partial_i+ \partial_i(a)$  for all
$a\in \F,\ \partial_i,\partial_j \in \Delta$ is called 
a {\bf (linear) differential ($\Delta$-) operator ring}.
\end{definition}
In particular, if derivation operators are trivial on $\F$, then $D$ is
isomorphic to the commutative polynomial ring with the same generators.

Every element $\sigma$ of $D$ may be uniquely represented as a
finite sum
$$
  \sigma=\sum_{\theta\in T(\Delta)}a_\theta\theta=\sum_{i_1,\dots,i_m}
   a_{i_1,\dots,i_m} \partial_1^{i_1},\dots,\partial_m^{i_m}.
$$

The maximal value of $\ord\theta$ among all $\theta$ for which $a_\theta\neq
0$, is called the {\bf order} of $\sigma$, it is 
denoted by $\ord\sigma$.

Let $D$~ be the ring of linear differential operators over the field 
$\F$. Consider on $D$ an ascending filtration
$(D_r)_{r\in\Z}$, where $D_r=\{f\in D\ |\ \ord f\leq r\}=
\F\cdot \Theta(r)$ for
$r\geq 0$, and $D_r=0$ for $r<0$.

By a {\bf filtered  $D$-module} we shall mean a $D$-module $M$ with
exhaustive and separable filtration $(M_r)_{ r\in\Z}$. It means that
$M=\bigcup_{r\in\Z}M_r$ and there exists $r_0\in\Z$ such that $M_r=0$ for all
$r<r_0$, $M_i\subseteq M_{i+1}$ and $D_i M_r\subseteq M_{r+i}$ for all
$r,i\in\Z$

\begin{definition}\label{D5.1.1}
Let $M$ be a filtered $D$-module with a filtration 
$(M_r)_{r\in\Z}$ and suppose that $M_r$ are finitely generated over 
$\F$ for any $r\in\Z$. Then we say that the filtration 
$(M_r)_{r\in\Z}$ is {\bf finite} and we call $M$ a
{\bf finitely filtered $D$-module}.

If there exists an integer $r_0\in\Z$ such that $M_s=D_{s-r_0}M_{r_0}$ for all $s>r_0$,
then the filtration $(M_r)_{r\in\Z}$ is called {\bf good}, and $M$ is 
called a {\bf good filtered $D$-module}.

A finite and good filtration of a $D$-module $M$ is called 
{\bf excellent}. In this
case $M$ is called an {\bf excellently filtered} $D$-module. 
\end{definition}

\begin{example}\label{5.1.7}
Let $M$ be a finitely generated $D$-module, and $\{m_i\}_{i\in I}$ be a
finite system of its generators. The filtration $M_r=\sum_{i\in I}D_rm_i$
is called {\bf associated} with these
generators. It is excellent.
\end{example}

\begin{example}\label{5.1.8}
Let $M$ be an excellently filtered $D$-module and $N$ be a 
submodule of $M$. Consider the
{\bf induced} filtration on $N$, $N_r=N\cap M_r$. 
According to a proposition 5.1.15 (see ~\cite{KLMP}),
the induced filtration also is excellent.
\end{example}

We will define now the {\bf Hilbert function } 
of filtered $D$-module as
$\chi(r) =\dim_{\F}M_r$. A next fact is well-known 
(see, for example, \cite{KLMP},
theorem 5.1.11). 
The characteristic Hilbert function of  excellently filtered
module 
for all sufficiently large $r\in\N$
is  the polynomial of degree less than or equal to $m$.
This numerical polynomial $\omega_M(s)$
is called 
{\bf Kolchin dimension polynomial}.
The degree $d=\deg(\omega_M) $ of Kolchin dimension polynomial  
is called
{\bf  a differential type}
of module $M$, the difference $(m - d)$ -- {\bf codimension},
and standard leading coefficient $a_{d}(\omega_M)$ -- {\bf  a typical
differential dimension.}

\begin{proposition}\label{rk}(see 5.2.12(\cite{KLMP}))
Let $\F$ be a differential field with a basic set $\Delta=\{\partial_1,\ldots,\partial_ m\}$,
$D$ be the ring of $\Delta$-operators over $\F$.
If $_DM$ is a finitely generated $D$-module, then for any excellent
filtration its 
$m$-standart coefficient $a_m(\omega_M)$ is equal to the maximal number of elements of $M$ 
which are linearly independent over $D$ (i.e. 
$a_m(\omega_M) =\rk_DM$.)
\end{proposition}

Let $M$ be a free $D$-module generated by $m_1,\dots, m_n$,
consider  the associated filtration 
(see an example \ref{5.1.7}).
Every element $f\in M$ can be expressed as
$f=\sum_{1\le j\le n}\sigma_jm_j$, where $\sigma_j\in D$. 
Set
$\ord_{m_j}f=\ord \sigma_j$ and $\ord f=\max_{1\le j\le n}
(\ord\sigma_j)$.

Let $N$ be a submodule of $M$, generated by elements
$\Sigma\subset N$ and
 $\ord_{m_j} f\le e_j$ for all  $j=1,\dots,n,\ f\in \Sigma$.
By the example \ref{5.1.8} the inducing filtration
$N$ is an excellent, therefore
the filtration of factor-module $M/N=(M_r/N_r)_{r\in\Z}$
is exellent also. Thus, there exists 
Kolchin's polynomial  $\omega_{M/N}$. Sometimes this polynomial 
is called the dimension polynomial
of system $\Sigma$ and is denoted by
$\omega_{[\Sigma]}$.

By the theorem  4.3.5\cite{KLMP}, 
using the theory of Gr\"obner basis,
we have for any orderly ranking on $M$
$\omega_{[\Sigma]}(s)=\sum_{j=1}^n \omega_{E_j}(s)$, 
where $E_j\subset \N_0^m$.
It easy to see, that if the system $\Sigma$ has a codimension 0,
its typical $\Delta$-dimension does not exceed $n$.

We are interested in following  
\begin{question}\label{1}
Let we know maximal orders $e_1,\dots, e_n$.
How to estimate a typical differential dimension $\Sigma$?
\end{question}
Firstly this question was asked by J.Ritt
for  ordinary differential systems. Later
E.Kolchin decided this problem in a codimension 1 even for nonlinear
systems.
\begin{theorem} (see \cite{Kolchin}, p.199)\label{Kol}
Let $\Sigma\subset\F\{y_1,\dots, y_n\}$, 
$\ord_{y_j} f\le e_j$ for all $f\in\Sigma,\ 1\le j\le n$
and $\rho$ be a prime component of $\{\Sigma\}$.
If the differential type of  $\rho$ is $m-1$, then  the
typical differential dimension $a_{m-1}$ of 
$\rho$ does not exceed  $e_1+ \dots +e_n$.
\end{theorem}
Note that for a system of linear differential equations an ideal
$\rho=[\Sigma]$  is prime, and Kolchin dimension polynomial
 $\omega_\rho$ coincides with a dimension polynomial
exellently filtered module of K\"ahler differentials. 
Below, for a  differential linear  system 
$\Sigma\subset\F\{y_1,\dots, y_n\}$,
we always will mean the exellently filtered module of differentials
$M/N$, equating $m_j$ with $\delta(y_j) $.

Now consider systems with  codimension 2.
\begin{theorem}(see 5.6.7, \cite{KLMP})\label{Bezy}
Let in the conditions of question \ref{1} $n=1$.
If  filtered $D$-module $M/N$ has a codimension 
 2, then
$a_{m-2}(\omega_{M/N})\le e_1^2$. 
This estimation is reachable.
\end{theorem}
Note that theorem ~\ref{Bezy}  generalizes classic
theorem of B\"ezout, which asserts that 
if derivation operators are trivial on
the field $\F$ (i.e. $D$ is a ring of commutative polynomials),
and all elements of $\Sigma$ are homogeneous, then 
$a_{d}\le h^{m - d}$, where
$d$ is a degree of characteristic Hilbert polynomial, 
$h=\max_{1\le j\le n}e_j$.

Note that in  \cite{Grig} was found the
following 
estimation for a typical differential dimension in any codimension: 
$a_{d}\le n(4m^2nh) ^{4^{m - d - 1}(2(m - d))}$.
However a lower estimation is unknown and whether 
exists a polynomial
bound is still open problem.

\section{Basic results.}

We are going to prove an estimation of a typical 
$\Delta$-dimension in a codimension 2 for $n>1$. 
At first, we  consider an example that gives a lower estimation.
\begin{example}\label{be1}
Consider the system $\Sigma$ of partial linear differential equations.
\begin{tabular}
{llllllll}
$\partial_1^{e_1} m_1=$&0;&&&&&&\\
$\partial_2^{e_1} m_1=$&$\partial_1^{e_2} m_2$;&&&&&&\\ 
&$\partial_2^{e_2} m_2=$&$\partial_1^{e_3} m_3$;&&&&&\\ 
&&$\partial_2^{e_3} m_3=$&$\partial_1^{e_4} m_4$;&&&&\\
&&$\dots$&&$\dots$;&&\\
&&&&$\partial_2^{e_i} m_i=$&$\partial_1^{e_{i+1}} m_{i+1};$&&\\
&&&&$\dots$&&$\dots$;\\
&&&&&&$\partial_2^{e_{n-1}} m_{n-1}=$&$\partial_1^{e_n} m_n$;\\ 
&&&&&&&$\partial_2^{e_n} m_n=0$.\\
\end{tabular}
\end{example}

\begin{proposition}
We have for the above system $\Sigma$ (see  example \ref{be1}):
$$\omega_{[\Sigma]}(s)=
\sum_{j_1+ \dots+ j_n=2
}e_{1}^{j_1}\dots e_{n}^{j_n}
\binom{s+ m - 2}{m - 2}
.$$
\end{proposition}
\begin{proof}
Consider the orderly ranking (see definition~\ref{ran}) 
such that $m_1>m_2>,\dots,>m_n$.
For the finding of characteristic set of the system $\Sigma$
(equations are linear,  therefore it is enough to compute the 
Gr\"obner basis),
calculate a critical pair of the first two equations. 
We get $\partial_1^{e_1+ e_2}m_2=0$.
Now we find a critical pair for this equation and 
third equation of the system.
We get an equations $\partial_1^{e_1+e_2+e_3}m_3=0$. For the last
generator will be got
$\partial_1^{e_1+\dots+e_n}m_n=0$.
By the theorem of 4.3.5\cite{KLMP} we have

\begin{align}
&
\omega_{[\Sigma]}(s)=
\omega_{\left(\begin{smallmatrix}e_1 &0&\dots& 0 \\0&e_1&\dots& 0 \end{smallmatrix}\right)}(s)
+\omega_{\left(\begin{smallmatrix}e_1+e_2& 0&\dots& 0 \\0 & e_2&\dots& 0\end{smallmatrix}\right)}(s)
+\dots\notag\\
&
+\omega_{\left(\begin{smallmatrix}e_1+e_2+\dots+e_n& 0&\dots& 0 \\0&  e_n&\dots& 0\end{smallmatrix}\right)}(s)
=e_1^2\binom{s+m-2}{m-2}+(e_1+e_2)e_2\binom{s+m-2}{m-2}+ \dots\notag\\
&
+(e_1+\dots+e_n)e_n\binom{s+m-2}{m-2}=
\sum_{j_1+\dots+j_n=2
}e_{1}^{j_1}\dots e_{n}^{j_n}
\binom{s+m-2}{m-2} \notag
\end{align}
\end{proof}

Thus, the bound of typical differential dimension in a codimension 2 
must to be not lower than
\begin{equation}\label{B}
\sum_{j_1+ \dots+ j_n =2}e_{1}^{j_1}\dots e_{n}^{j_n}.
\end{equation}
This example supports formulated  
 in \cite{KLMP} (see a formula~ (5.6.4)) 
conjecture. It was disproved 
(see an example~5.6.6)  
for a codimension more than 2.
Still it is unknown, whether this conjecture is true in a 
codimension 2.

Now we will prove an upper bound for typical $\Delta$-dimension
that also, as in (\ref{B}), is squared by
orders of the equations in the system $\Sigma$.
We will prove such bound
$$
a_{m-2}(\omega_\Sigma)\le 2^{2(m+1)}(e_1+\dots+e_n)^2.
$$
It is known that if the field $\F$
contains the field of rational functions $\C(x_1,\dots, x_m)$, 
and  $\partial_j(x_j)=1$, then for every $\Delta$-extension $\F$ of 
positive codimension  there exists $\Delta$-primitive element
(see for example, \cite{KLMP}, 5.3.13).

We will prove the constructive variant of this theorem for linear
equations. It has an independent interest.  Namely, in positive codimension 
the linear system in $n$ indeterminates of
order  not  greater than $h$ is
equivalent to the linear system  in one indeterminate of order not
greater than $O(m)(n+1)h$.

\begin{theorem}\label{prim}
Let $\Sigma\subset\F\{y_1,\dots, y_n\}$ be the system of linear
differential equations, $\ord f_{y_j}\le e_j$ for all
$f\in\Sigma$, $1\le j\le n$ and $a_m(\omega_{[\Sigma]}) =0$.

Then in some extension $\G$ of the field $\F$ exist
elements $c_2,\dots, c_n$ such, that module of differentials
$M/N$ of systems $\Sigma$ is generated by one element $M/N=\tilde D\psi$,
where $\psi=m_1+c_2m_2 +\dots +c_nm_n$,
$\tilde D=\G\otimes_\F D$. Denote
$\lambda_j\in\tilde D:\ \lambda_j \psi=m_j$.
Then for any  
$\lambda_j$ we have 
$$
\ord\lambda_j\le 2^{m}(e_1+ \dots+ e_n).
$$
\end{theorem}
\begin{proof}
Since a codimension of the system $\Sigma$ is greater than  0, the
rank of $D$-module differentials $M/N$ is equal to 0.
It is clear, that $\rk_D N=n$, and we can choose in 
$\Sigma$ independent over $D$ subsystem
$\Sigma'$
such that $\Card \Sigma'=n$. Let $\Sigma'=\{F_1,\dots, F_n\}$.
Denote by $\Sigma_0$ following system
\begin{align}
F_i(m_1,...,m_n) &=0;\ i=1,\dots, n,\ F_i\in\Sigma'\notag\\
m_1+c_2m_2+ \dots+ c_nm_n&=\psi. \notag
\end{align}
Here $c_i, \psi$ are new differential indeterminates.
Consider $\Sigma_0$ as a system of linear equations relatively
$\Theta m_j$ with coefficients in the field
$\F'=\F(\Theta(c_2),\dots,\Theta(c_n),\Theta(\psi))$.
$\Sigma_0$ is the linear system of $n+1$  independent equations   in
$k(\Sigma_0)=\binom{e_1+m}{m}+\dots+\binom{e_n+m}{m}$ indeterminates over 
$\F'$.
It means that
the rang of $(n+1)\times k(\Sigma_0)$-matrix 
of corresponding  homogeneous system
is maximal and equals to $n+1$.
Now add the derivations of 
 $\Sigma_0$.  Let $\Sigma_s$ be
the system
\begin{align}
\Theta(s) F_i(m_1,...,m_n) &=0;\ i=1,\dots, n,\ F_i\in\Sigma'\notag\\
\theta(m_1+ c_2m_2 \dots c_nm_n) &=\theta(\psi), \theta\in\Theta(s). \notag
\end{align}
$\Sigma_s$ is a system in $\binom{s+e_1+m}{m}+\dots+\binom{s+e_n+ m}{m}$ indeterminates,
and 
the number of independent equations equals to 
$(n+1) \binom{s+m}{m}$.
We see that the the number of equations grows quicker, 
than number  of indeterminates. From independence  of 
$\Sigma'$ follows, that there exists
such $s$, that $\Sigma_s$ is   uniquely determinated. 
The matrix of the linear system $\Sigma_s$
consists of elements from the field $\F(\theta(s) (c_1),\dots,\theta(s) (c_n),\theta(s) (\psi))$.
Without using a division, we may transform this matrix  to the 
triangular form. It means, that we obtain
for any
 $j=1,\dots,n$ the expression $\alpha_jm_j=\lambda_j\psi$, where
$\alpha_j=\sum_{i=2}^n\sigma_i(c_i)$, $\sigma_i\in D_s$ 
is a derivative operator  of order  $\le s$.
Now it is sufficiently  to join to the differential field $\F$ 
the elements $c_2,\dots,c_n$, satisfying the conditions
$\alpha_i(c_2,\dots,c_n)\neq 0$.
Let $\G=\F\langle c_2,\dots,c_n\rangle$, 
$\tilde D=\G\otimes_\F D$. We see, that $\tilde D$-module $M/N$ 
is generated by $\psi=m_1+c_2m_2+\dots +c_nm_n$,
and any $m_j$ can be expressed as 
 $\lambda_j\psi, \ord\lambda_j\le s$.  

We must
find a suitable value $s$.
Show that for $s=2^m(e_1+ \dots+ e_n)$  a condition 
$$\binom{s+ e_1+m}{m}+ \dots+ \binom{s+e_n+m}{m}\le(n+1)
\binom{s+m}{m}$$
holds. Actually,
\begin{align}
\frac{\binom{s+m+e_1}{m}+\dots+\binom{s+m+e_m}{m}}
{\binom{s+m}{m}}=
\frac{\prod_{j=1}^m(s+e_1+j)}{\prod_{j=1}^m(s+j)}+\dots+
\frac{\prod_{j=1}^m(s+e_n+j)}{\prod_{j=1}^m(s+j)}=\notag\\
\prod_{j=1}^m\frac{(s+e_1+j)}{(s+j)}+\dots+
\prod_{j=1}^m\frac{(s+e_n+j)}{(s+j)}=
\prod_{j=1}^m\big(1+\frac{e_1}{s+j}\big)+\dots+
\prod_{j=1}^m\big(1+\frac{e_n}{s+j}\big)\le
\notag\\
\big(1+\frac{e_1}{s+1}\big)^m+\dots+\big(1+\frac{e_1}{s+1}\big)^m=
n+\sum_{j=1}^m\binom{m}{j}\big({\frac{e_1}{s+1}}\big)^j+\dots+
\sum_{j=1}^m\binom{m}{j}\big({\frac{e_n}{s+1}}\big)^j.\notag
\end{align}
Thus, for $s+1\ge\max_{i=1}^ne_i$ we have
\begin{align}
\frac{\binom{s+m+e_1}{m}+\dots+\binom{s+m+e_m}{m}}
{\binom{s+m}{m}}\le
&n+(2^m-1)(\frac{e_1}{s+1}+\dots+\frac{e_n}{s+1})\le \notag\\
&n+2^m\frac{(e_1+\dots+e_n)}{s+1}\le n+1,
\notag
\end{align}
since we put $s=2^m(e_1+\dots+e_n)$.
\end{proof}
As follows from the proof of theorem \ref{prim}, 
to find a primitive element
it is enough 
to differentiate
necessary number of  times the system 
$\Sigma'$ and to eliminate from
got linear system variables (for example, by the Gauss' method).
Thus, the theorem gives the algorithm to find a primitive element
for systems of linear differential equations.

Now we will prove the analogue of B\'ezout theorem  in the case 
of codimension 2 for the systems of linear differential equations.
\begin{theorem}
Let $\Sigma\subset\F\{y_1,\dots, y_n\}$ be a system of linear partial 
differential equations, $m=\Card\Delta$,
and let $\ord_{y_j} f\le e_j$ for all $f\in\Sigma,\ 1\le j\le n$.
Suppose that system $\Sigma$ has
differential type $m-2$. Then its
typical differential dimension $a_{m - 2}$ does not exceed 
\begin{equation}\label{bound}
2^{2m+2}(e_1+\dots+e_n)^2.
\end{equation} 
\end{theorem}
\begin{proof}
By the theorem \ref{prim} we may suppose that
 module of differentials $M/N$ is
generated by a primitive element $\psi$. 
Moreover, $m_j$ is expressed 
as the differential operators of order $\le 2^{m}(e_1+ \dots+e_n)$ in  $\psi$.
Let $J=\{\lambda\in\tilde D:\ \lambda \psi=0\}$ 
be an annihilator of element $\psi$.
$J$ is generated as an ideal in the ring of
differential operators $\tilde D$ by the elements of order not greater 
than
$$
2^{m}(e_1+\dots+e_n)+\max_{j=1}^n e_j,
$$
since we obtain the same orders of operators 
after  substitution the
expressions of $m_i$ as $\lambda_j\psi$ in the system $\Sigma$.
Thus, $J$ is 
generated  by the elements of order
$\le 2^{m}(2e_1+ \dots+ 2e_n)$. 
Now by  theorem \ref{Bezy}
we get an estimation(\ref{bound}).
\end{proof}
So, in the case of the system of linear differential
equations we got an upper and lower squared bound of a typical
differential dimension for differential type
$m-2$. It is better, than   estimation \cite{Grig}.

\end{document}